\documentclass[12pt]{amsart}%{article}%
\usepackage{amssymb}
\usepackage[mathscr]{eucal}
\usepackage{epsf}
\usepackage{epsfig}
\usepackage{color}
\DeclareGraphicsExtensions{.pstex,.eps,.epsi,.ps}

 \newtheorem{thm}{Theorem}

 \theoremstyle{definition}
 
 \theoremstyle{remark}
 
 \numberwithin{equation}{section}
 \newcommand{\Real}{\mathbb{R}}

  \newcommand{\gl}{\mathfrak{gl}}
\newcommand{\GL}{\text{GL}}
\begin{document}
\title[Matrix Factorizations]{Matrix Factorizations via the Inverse Function Theorem}

\author{Paul W.Y. Lee}
\email{wylee@math.cuhk.edu.hk}
\address{Room 216, Lady Shaw Building, The Chinese University of Hong Kong, Shatin, Hong Kong}
\date{\today}

\maketitle

\begin{abstract}
We give proofs of QR factorization, Cholesky's factorization, and LDU factorization using the inverse function theorem. As a consequence, we obtain analytic dependence of these matrix factorizations which does not follow immediately using Gaussian elimination.
\end{abstract}

\section{Introduction}

In this note, we give proofs, using the inverse function theorem, some classical results in linear algebra. This includes QR factorization, Cholesky's factorization, and LDU factorization. Let us take QR factorization as an example. In this case, we define a map 
\[
F(Q,R)=QR
\]
where $Q$ ranges over all orthogonal matrices and $R$ ranges over all upper triangular matrices with non-negative diagonal entries. 

We first show that the derivative of $F$ at a point $(Q,R)$ is an isomorphism whenever $R$ is invertible. Therefore, the inverse function theorem applies at those points. Then it is not hard to show that every invertible matrix has a unique QR factorization. A limiting argument shows any square matrix has a QR factorization. Since the map $F$ is analytic, it also shows that the factorization depends analytically on the matrices (see Theorem \ref{QR} for the precise statement). All the proofs in this note follow this pattern.

We hope that this note can shed new light on finding new useful matrix factorizations and on the generalizations of the above matrix factorizations to infinite dimensions.

\smallskip

\section{QR factorization}

A QR factorization of a real square matrix $A$ is a decomposition of $A$ into the product of an orthogonal matrix and an upper triangular matrix. More precisely, let $\mathcal Q$ be the set of all $n\times n$ orthogonal matrices and let $\mathcal R^+$ be the sets of all $n\times n$ upper triangular matrices with positive diagonal entries, respectively. Let $F:\mathcal Q\times\bar{\mathcal R}^+\to \gl(n,\Real)$ be the map defined by
\[
F(Q,R)=QR,
\]
where $\bar{\mathcal R}^+$ denotes the closure of $\mathcal R^+$ and $\gl(n,\Real)$ denotes the set of all $n\times n$ matrices. Of course, the fact that every real square matrix has a QR factorization means the map $F$ is surjective.

\begin{thm}[QR factorization]\label{QR}
The map $F$ satisfies the followings.
\begin{enumerate}
\item The map $F$ is proper (i.e. inverse image of compact sets are compact).
\item The restriction of $F$ to the set $\mathcal Q\times \mathcal R^+$ is an analytic diffeomorphism onto its image $F(\mathcal Q\times \mathcal R^+)=\GL(n,\Real)$, the general linear group.
\item All $n\times n$ matrices $A$ can be written as a product of an orthogonal matrix and an upper triangular matrix. This decomposition is unique if $A$ is invertible and all the diagonal entries of the upper triangular matrix are positive.
\end{enumerate}
\end{thm}

\begin{proof}
First, we prove that $F$ is proper. For this, let $(Q_i,R_i)$ be a sequence in $\mathcal Q\times\bar{\mathcal R}^+$ such that $||(Q_i,R_i)||\to\infty$ as $i\to\infty$. Here and throughout this paper, we use the Hilbert-Schmidt norm. It is enough to show $||F(Q_i,R_i)||\to\infty$ as $i\to\infty$.
Since $\mathcal Q$ is compact, we have $||R_i||\to\infty$. It follows that $||F(Q_i,R_i)||=||Q_iR_i||=||R_i||\to\infty$ as $i\to\infty$. This proves (1).

Let $Q$ be an orthogonal matrix and let $R$ be an upper triangular matrix with positive diagonal entries. The derivative $DF(Q,R)(U,V)$ of the function $F$ at the point $(Q,R)$ is given by
\[
DF(Q,R)(U, V)=U R+QV.
\]
Note that since $V$ is tangent to $\mathcal R^+$, it is upper triangular. Since $U$ is tangent to $\mathcal Q$ at the point $Q$, we also have $U^TQ+Q^TU=0$. In other words, $Q^TU$ is skew symmetric. Therefore,
$U R+QV=0$ if and only if
\[
Q^TU=-VR^{-1}.
\]
The left hand side of the above equation is skew symmetric while the right hand side is upper triangular. Therefore, we must have $U=V=0$ and $DF(Q,R)$ is an isomorphism. It follows from a dimension count and the inverse function theorem that $F$ restricted to $\mathcal Q\times\mathcal R^+$ is an analytic local diffeomorphism.

Let $(Q',R')$ be another point in the space $\mathcal Q\times\mathcal R^+$ such that $F(Q,R)=F(Q',R')$. Then
\[
Q^TQ'=R(R')^{-1}.
\]
The left hand side of the above equation is orthogonal and the right hand side is upper triangular with positive diagonal entries. It follows that $Q^TQ'$ is an identity matrix. Therefore, $Q=Q'$ and $R=R'$. This shows that the restriction of $F$ to $\mathcal Q\times\mathcal R^+$ is injective and hence an analytic diffeomorphism onto its image. Therefore, it remains to show that $F(\mathcal Q\times \mathcal R^+)=\GL(n,\Real)$. For this, note that $\GL(n,\Real)$ contains two connected components one containing the identity $I$ and the other containing $-I$. Since both $I$ and $-I$ are clearly contained in the image $F(\mathcal Q\times \mathcal R^+)$, it is enough to show that there is no boundary point of $F(\mathcal Q\times \mathcal R^+)$ contained in $\GL(n,\Real)$.

Let $A$ be a point which is in the boundary of $F(\mathcal Q\times \mathcal R^+)$. Then there is a sequence $(Q_i,R_i)$ such that $F(Q_i,R_i)$ converges to $A$ as $i\to\infty$. Since $F$ is proper, a subsequence of $(Q_i, R_i)$ converges to a limit $(Q_\infty,R_\infty)$ which is a QR factorization of $A$. Therefore, if $A$ is in $\GL(n,\Real)$, then $R_\infty$ is invertible. Hence, $(Q_\infty, R_\infty)$ is contained in $\mathcal Q\times \mathcal R^+$. This gives a contradiction since $A$ is a boundary point of the open set $F(\mathcal Q\times \mathcal R^+)$. This shows $F(\mathcal Q\times \mathcal R^+)=\GL(n,\Real)$ and hence finishes the proof of (2). The above argument also shows that any matrix $A$ has a QR factorization proving (3).

\end{proof}

\smallskip

\section{Cholesky's factorization}

Every non-negative definite symmetric matrix $A$ has a decomposition of the form $A=LL^T$, where $L$ is a lower triangular matrix. This is the, so called, Cholesky decomposition. Let $\mathcal L^+$ be the set of all $n\times n$ lower triangular matrices with positive diagonal entries and let $\mathcal S^+$ be the set of all $n\times n$ positive definite symmetric matrices. Let $F:\bar{\mathcal L}^+\to\bar{\mathcal S}^+$ be the map defined by $F(L)=LL^T$.

\begin{thm}[The Cholesky factorization]
The map $F$ satisfies the followings.
\begin{enumerate}
\item The map $F$ is proper.
\item The restriction of $F$ to the set $\mathcal L^+$ is an analytic diffeomorphism onto its image $F(\mathcal L^+)=\mathcal S^+$.
\item Any non-negative symmetric matrix $A$ can be written as $LL^T$, where $L$ is lower triangular with non-negative diagonal entries. Moreover, such decomposition is unique if $A$ is positive definite.
\end{enumerate}
\end{thm}

\begin{proof}
The proof is similar to that of Theorem \ref{QR}. For (1), let $L_i$ be a sequence of lower triangular matrix such that $||L_i||\to\infty$ as $i\to\infty$. It follows from H\"older's inequality that
\[
||F(L_i)||^2=||L_iL_i^T||^2\geq \frac{1}{n}(\text{tr}(L_iL_i^T))^2\to\infty
\]
as $i\to\infty$. Therefore, (1) follows.

For (2), we look at the derivative $DF$ of $F$ at a point $L$ in the space $\mathcal L^+$. It is given by
\[
DF(L)(V)=LV^T+VL^T.
\]
Since $V$ is tangent to the space $\mathcal L^+$, it is lower triangular.
If $LV^T+VL^T=0$ holds, then
\[
V^T(L^T)^{-1}=-L^{-1}V.
\]
The left side of the above equation is an upper triangular matrix while the right side is a lower triangular one. This implies that
\[
V=LD
\]
for some diagonal matrix $D$. By substituting this back, we obtain $D=0$ and hence $V=0$. This shows that $DF(L)$ is an isomorphism for each $L$ in $\mathcal L^+$. Therefore, the inverse function theorem applies and the restriction of $F$ to $\mathcal L^+$ is a local diffeomorphism.

Let $L'$ be a point in $\mathcal L^+$ such that $F(L)=F(L')$. Then
\[
L^T((L')^{-1})^T=L^{-1}L'.
\]
The left hand side of the above equation is upper triangular and the right side is lower triangular. It follows that
\[
L'=LD\text{\quad and \quad } L=L'D
\]
for some diagonal matrix $D$ with positive diagonal entries. Since $L$ is invertible, we have $D=I$ and so $L=L'$. Therefore, the restriction of $F$ to $\mathcal L^+$ is injective. Now (2) and (3) follow as in the proof of Theorem \ref{QR}.

\end{proof}

\smallskip

\section{LDU factorization}

In this section, we discuss the LDU factorization. It turns out that this case is slightly more complicated since the corresponding map here is not proper. For simplicity, we only consider the case when all the matrices are invertible.

Let $\mathcal L_1$ denotes the set of all $n\times n$ lower triangular matrices for which diagonal entries are all 1, $\mathcal U_1$ denotes the set of all $n\times n$ upper triangular matrices for which diagonal entries are all 1, and let $\mathcal D$ denotes the set of all invertible diagonal matrices. Let $F:\mathcal L_1\times\mathcal D\times\mathcal U_1\to \text{GL}(n,\Real)$ be the map defined by $F(L,D,U)=LDU$.

Finally, recall that the leading principal minors of a square matrix $A$ are matrices obtained by deleting the last $k$ rows and the last $k$ columns for some non-negative integer $k$. We denote by $\mathcal P$ the set of all matrices $A$ for which all leading principal minors of $A$ are invertible. It is known that an invertible matrix has a LDU factorization if and only if it is contained in $\mathcal P$ (see, for instance, \cite{HoJo}). Moreover, the factorization is unique. To see that the image  $F(\mathcal L_1\times\mathcal D\times\mathcal U_1)$ of $F$ is contained in $\mathcal P$, let $(L,D,U)$ be a point in the set $\mathcal L_1\times\mathcal D\times\mathcal U_1$. Let us split $L=\left(
                                      \begin{array}{cc}
                                        L^1 & O \\
                                        L^2 & L^3 \\
                                      \end{array}
                                    \right)$, $D=\left(
                                      \begin{array}{cc}
                                        D^1 & O \\
                                        O & D^4 \\
                                      \end{array}
                                    \right)$,  and $U=\left(
                                      \begin{array}{cc}
                                        U^1 & U^2 \\
                                        O & U^4 \\
                                      \end{array}
                                    \right)$,
where the upper left block of each matrix is of size $k\times k$. Then $LDU=\left(
                                      \begin{array}{cc}
                                        L^1D^1U^1 & L^1D^1U^2 \\
                                        L^2D^1U^1 & L^2D^1U^2+L^3D^4U^4 \\
                                      \end{array}
                                    \right)$. Therefore, $L^1D^1U^1$ is the $k\times k$ leading principal minor of $LDU$ and it is invertible if $D$ is. Hence, $LDU$ is contained in $\mathcal P$.

\begin{thm}[LDU factorization]
The followings hold.
\begin{enumerate}
\item The map $F$ is an analytic diffeomorphism between $\mathcal L_1\times\mathcal D\times\mathcal U_1$ and $\mathcal P$.
\item Every matrix in $\mathcal P$ has a unique LDU factorization.
\end{enumerate}
\end{thm}

\begin{proof}
Let $(L,D,U)$ be a point in $\mathcal L_1\times\mathcal D\times\mathcal U_1$. If
\[
DF(L,D,U)(A,S,B)=ADU+LSU+LDB=0,
\]
then $L^{-1}AD+S+DBU^{-1}=0$. Since $L^{-1}AD$ is strictly lower triangular, $DBU^{-1}$ is strictly upper triangular, and $S$ is diagonal, we must have $L^{-1}AD=DBU^{-1}=S=0$. Since $D$ is invertible, we must have $A=B=0$. This shows that $F$ is a local diffeomorphism.

Suppose $F(L,D,U)=F(L',D',U')$. Then $(L')^{-1}LD=D'U'U^{-1}=:\bar D$ is diagonal. Since diagonals of both matrices $(L')^{-1}L$ and $U'U^{-1}$ are all 1's, we have $D=D'=\bar D$. It also follows from this and $(L')^{-1}L=I$ that $L=L'$. Similarly $U=U'$. Therefore, $F$ is injective.

Let $A$ be a point in $\mathcal P$ which is also contained in the boundary of $F(\mathcal L_1\times\mathcal D\times\mathcal U_1)$. Let $(L_i,D_i,U_i)$ be a sequence such that $F(L_i,D_i,U_i)$ converges to $A$ as $i\to\infty$. We want to show that $(L_i,D_i,U_i)$ stays bounded as $i\to\infty$. The determinants of the leading principal minors of $L_iD_iU_i$ are converging to those of $A$ which are nonzero. By the discussion at the beginning of the section, we know that the determinant of leading principal minors of $L_iD_iU_i$ are given by
\[
D_{11}^i, \,D_{11}^iD_{22}^i,...,\,D_{11}^i...D_{nn}^i,
\]
where $D_{jj}^i$ is the $j$-th diagonal entry of $D^i$. It follows that $D_i$ converges to a diagonal matrix with non-zero diagonal entries. Let $\left(
                                      \begin{array}{cc}
                                        L_i^1 & O \\
                                        L_i^3 & 1 \\
                                      \end{array}
                                    \right)$, $\left(
                                      \begin{array}{cc}
                                        D_i^1 & O \\
                                        O & D_i^4 \\
                                      \end{array}
                                    \right)$,  and $\left(
                                      \begin{array}{cc}
                                        U_i^1 & U_i^2 \\
                                        O & 1 \\
                                      \end{array}
                                    \right)$  be the upper left $k\times k$ blocks of $L_i$, $D_i$, and $U_i$, respectively. Here $L_i^1$, $D_i^1$, and $U_i^1$ are of size $(k-1)\times(k-1)$. Then
\[
L_iD_iU_i=\left(
                                      \begin{array}{cc}
                                        L_i^1D_i^1U_i^1 & L_i^1D_i^1U_i^2 \\
                                        L_i^3D_i^1U_i^1 & L_i^3D_i^1U_i^2+D_i^4 \\
                                      \end{array}
                                    \right).
\]
By induction, we can assume that $L_i^1$ and $U_i^1$ are bounded independent of $i$. On the other hand, $L_i^1D_i^1U_i^2$ and $L_i^3D_i^1U_i^1$ are converging to the corresponding blocks of $A$. Since $L_i^1D_i^1$ and $D_i^1U_i^1$ are invertible, $L_i^3$ and $U_i^2$ are bounded independent of $i$ as well. By induction, this shows that $L$ and $U$ are bounded. Therefore, we can extract a convergence subsequence of $(L_i,D_i,U_i)$ which converges to a point $(L_\infty,D_\infty,U_\infty)$ satisfying $A=L_\infty D_\infty U_\infty$. Since $A$ is invertible, so is $D_\infty$. Therefore, $(L_\infty,D_\infty,U_\infty)$ is contained in $\mathcal L_1\times\mathcal D\times\mathcal U_1$ which is a contradiction. Therefore $F$ is surjective. This finishes the proof of (1). (2) follows immediately from (1).
\end{proof}

\end{document}